\documentclass[12pt]{amsart}
\usepackage{amssymb,amsfonts,amsmath,amsopn,amstext,amscd,color,latexsym,mathrsfs,skak,stackrel,url,verbatim}
\usepackage{mathabx}
\usepackage[utf8]{inputenc}
\usepackage[all, cmtip]{xy}

\usepackage[active]{srcltx}

\usepackage[pdftex, colorlinks=true, linktoc=all]{hyperref}

\usepackage{cleveref}

\numberwithin{equation}{subsection}

\newtheorem{theorem}[equation]{Theorem}
\newtheorem{lemma}[equation]{Lemma}
\newtheorem{corollary}[equation]{Corollary}

\theoremstyle{definition}

\newtheorem{example}[equation]{Example}

\theoremstyle{remark}
\newtheorem{remark}[equation]{Remark}
\newtheorem{case}[equation]{Case}
\newtheorem{subcase}[equation]{Subcase}
\theoremstyle{observation}

\newcommand{\hide}[1]{}

\newcommand{\bbC}{{\mathbb C}}

\newcommand{\bbH}{{\mathbb H}}

\newcommand{\bbQ}{{\mathbb Q}}

\newcommand{\bbZ}{{\mathbb Z}}

\newcommand{\rH}{{\rm H}}

\newcommand{\rR}{{\rm R}}

\begin{document}
\title[KV Formulation of Log Akizuki-Nakano Vanishing]{A Kawamata-Viehweg type formulation of the Logarithmic Akizuki-Nakano Vanishing Theorem}
\author{Donu Arapura}
\address{Department of Mathematics, Purdue University,
150 N. University Street, West Lafayette, IN 47907, U.S.A.}
\email{arapura@math.purdue.edu}
\author{Kenji Matsuki}
\address{Department of Mathematics, Purdue University,
150 N. University Street, West Lafayette, IN 47907, U.S.A.}
\email{matsuki@purdue.edu}
\author{Deepam Patel}
\address{Department of Mathematics, Purdue University,
150 N. University Street, West Lafayette, IN 47907, U.S.A.}
\email{patel471@purdue.edu}
\author{Jaros{\l}aw W{\l}odarczyk}
\address{Department of Mathematics, Purdue University,
150 N. University Street, West Lafayette, IN 47907, U.S.A.}
\email{wlodarcz@purdue.edu}

\thanks{D.A. was partially supported by a grant from the Simons
  foundation. K.M. would like to thank RIMS at Kyoto for their warm hospitality and generous support. D.P. would like to acknowledge support from the National Science Foundation award DMS-1502296.  J.W. would like to thank the Binational (U.S. - Israel) Science Fund BSF 2014365 for their support. }

\begin{abstract}
In this largely expository article, we present a Kawamata-Viehweg type formulation of the (logarithmic) Akizuki-Nakano Vanishing Theorem. While the result is likely known to the experts, it does not seem to appear in the existing literature. 
\end{abstract}
\maketitle

\tableofcontents

\section{Introduction}

In the following, we work over the field ${\mathbb C}$ of complex numbers. However, by the Lefschetz principle or flat base change, all the vanishing results in this paper are valid over any field of characteristic zero.\\

Let $X$ be a smooth projective variety, $A$ an integral ample divisor on $X$, and $K_X$ the canonical divisor. In this setting, the classical Kodaira Vanishing Theorem (\cite{Kodaira}) states that
\begin{equation*}\label{KV}
\rH^i(X,{\mathcal O}_X(K_X + A)) = 0 \text{ for }i > 0.
\end{equation*}

According to the Iitaka philosophy (cf. \cite{Matsuki}), we obtain its logarithmic version (\cite{Norimatsu}) by adding a simple normal crossings divisor $D = \sum D_i$, called the bounday, on $X$:
\begin{equation}\label{KVlog}
\rH^i(X,{\mathcal O}_X(K_X + D+A)) = 0 \text{ for }i > 0.
\end{equation}

The celebrated Kawamata-Viehweg Vanishing Theorem (\cite{Kawamata, Viehweg}) further generalizes \eqref{KVlog} to the setting where $A$ is a $\bbQ$-divisor, by allowing the boundary divisor to have a fractional part $F = \sum f_jF_j \hskip.1in (0 < f_j < 1)$ such that $F + A$ is integral (i.e.,  $F = \lceil A\rceil - A$) as well as an integral part $B = \sum_kB_k$.  It states that
\begin{equation*}\label{KV}
\rH^i(X,{\mathcal O}_X(K_X + B + F + A)) = \rH^i(X,{\mathcal O}_X(K_X + B + \lceil A\rceil) = 0 \text{ for }i > 0,
\end{equation*}
where $B$ and $F$ share no common components, and $\text{Supp}(B \cup F)$ is a simple normal crossings divisor.\\

Given the preceding discussion, it is natural to ask for an analog of the previous picture in the setting of the Akizuki-Nakano Vanishing.
More precisely, recall that the Akizuki-Nakano Vanishing Theorem(\cite{Akizuki-Nakano}) states that
\begin{equation}\label{AKvanishing}
\rH^i(X,\Omega_X^j( A)) = 0 \text{ for }i + j > \dim X,
\end{equation}
where $A$ is an integral ample divisor on $X$.
As before, the Iitaka philosophy suggests that one would obtain a logarithmic version of \eqref{AKvanishing} by considering a simple normal crossings boundary divisor $D = \sum D_i$, and replacing the usual sheaf of differential forms with the sheaf of logarithmic differential forms. This leads precisely to the Esnault-Viehweg Vanishing Theorem (\cite{EVB, EVI}):
\begin{equation*}\label{EVvanishing}
\rH^i(X,\Omega_X^j(\log(D))( A)) = 0 \text{ for }i + j > \dim X.
\end{equation*}
At about  the same time, Steenbrink \cite{Steenbrink} proved that 
\begin{equation*}\label{Stvanishing}
\rH^i(X,\Omega_X^j(\log(D))( A-D)) = 0 \text{ for }i + j > \dim X.
\end{equation*}
Sometime later the first author \cite{Arapura2} found a ``fractional''  version, which will be explained in section \ref{section:pvanishing}.
The statement of this fractional version does not seem to yield directly, as its special cases, all the known classical vanishing theorems mentioned above. However, a slight modification of the  statement does, and this is the main result that we want to explain.
 Suppose that $F$ is a fractional divisor with support contained in $D$, and let $G$ be an integral divisor such that $F \leq G \leq D$.  
Then we will show in the text that
\begin{equation*}
\rH^i(X,\Omega_X^j(\log(D))(F + A - G)) = \rH^i(X,\Omega_X^j(\log(D))(\lceil A\rceil - G) = 0
\text{ for } i + j > \dim X.
\end{equation*}

%
%As a step towards this formulation, the first author provided a statement which can be seen as a ``fractional'' version of the vanishing theorem due to Steenbrink (\cite{Arapura2, Steenbrink}).  This statement does not seem to yield directly, as special cases, all the known classical vanishing theorems mentioned above. However, a slight modification to this statement does yield all the known classical vanishing theorems as its special cases. The presentation of this formulation is the main theme of this paper.  (We note that the fractional version of the Steenbrink Vanishing Theorem does imply our formulation via some ``round up'' trick.  See \ref{helmke}.)\\

This statement should not come as a surprise to the experts. For
example, it readily follows, via some standard arguments, from theorem
6.2  of the beautiful book by Esnault-Viehweg (\cite{EVB}).  We also recently
learned of a nice paper by C. Huang,  K.  Liu, X. Wan, and X. Yang
\cite{HLWY}, which proves a similar result by $L^2$-methods in the K\"ahler setting.
Other authors have also considered certain cases in the
presence of singularities (e.g. \cite{Kov2}).  
We do not consider such cases in this article.  We are not claiming much originality in the result. Our goal here is to present a Kawamata-Viehweg type formulation of the Kodaira-Akizuki-Nakano Vanishing Theorem in a way that is easily accessible to non-experts. \\
%While we do not claim any orignality, we present two proofs of different flavors.  \\

We present two proofs of the main result.
The first is elementary,  and involves a reduction to  Steenbrink's Vanishing Theorem by the  Kawamata Covering Lemma.  Furthermore, since the logarithmic differential forms have no ramification under the Kummer covers, this proof makes it clear that the process of taking the round up ``$\lceil {A}\rceil$'' does not stem from the ramification, but rather from the subtraction of some effective divisor ``$G$'' appearing in our formulation. We note that this fact is not readily visible in the classical proof of the Kawamata-Viehweg Vanishing Theorem (\cite{Kawamata,KMM,Matsuki}) when reducing the fractional case to the integral case via the covering technique. This is caused by the fact that the Kawamata-Viehweg Vanishing Theorem only deals with top degree forms, while the Akizuki-Nakano Vanishing Vanishing Theorem deals also with lower degree differential forms.  We emphasize again that all the essential ideas of reducing the fractional case to the integral case via covering already appear in \cite{EVB} as well as in \cite{Kawamata}. \\

The second proof uses a simplified version of an argument in \cite{Arapura2} to  establish  a fractional form of the Steenbrink Vanishing theorem.  It uses the method of Deligne-Illusie \cite{DI} in positive characteristic. It is also worth noting that instead of the Kawamata Covering Lemma, it uses a lemma of Hara \cite{Hara} to handle the fractional parts. Once the fractional version of the Steenbrink Vanishing is proved, our main result follows immediately as an easy corollary via some ``round up'' tricks in \ref{helmke}.\\

%is based upon the results by Deligne-Illusie \cite{DI}, and extracts the ideas in \cite{Arapura1,Arapura2}.  This proof verifies the fractional version of the Steenbrink Vanishing directly, via the lamma by Hara \cite{Hara}, analyzing the behavior of certain sheaves under the Frobenius map.  Once the fractional version of the Steenbrink Vanshing is proved, our formulation follows immediately as an easy corollary via some ``round up'' trick \ref{helmke}.\\

We now briefly outline the contents. In the next section, we state our Kawamata-Viehweg type formulation of the logarithmic Akizuki-Nakano Vanishing Theorem, and explain how to obtain the classical vanishing theorems discussed above as special cases. We also discuss the failure of some naive versions of such a formulation. In the third section, we present the first proof of the main result with some remarks and an alternate argument.  In the fourth section, we then present the second proof of the main result.  Finally, in the last section, we discuss some potential applications and further generalizations. 

\section{The main vanishing result.}

\subsection{Statement of the main vanishing result}
Let $X$ be a smooth projective variety, $D = \sum D_i$ a simple normal crossings divisor, $A$ an ample ${\mathbb Q}$-divisor, and
$F: = \lceil A\rceil - A$.

\begin{theorem}\label{mainvan} Let $X, D, A,$ and $F$ be as above. Suppose that $F = \lceil A\rceil - A \leq D$ (i.e. the support of $F$ is contained in $D$), and let $G$ be an integral divisor such that $F \leq G \leq D$.  
Then we have:
\begin{equation*}
\rH^i(X,\Omega_X^j(\log(D))(F + A - G)) = \rH^i(X,\Omega_X^j(\log(D))(\lceil A\rceil - G) = 0
\text{ for } i + j > \dim X.
\end{equation*}
\end{theorem}

The following relative version easily follows from the above absolute version (see, e.g., an argument in the proof of Theorem 1-2-3 \cite{KMM}).

\begin{corollary}\label{relmainvan} Let $f:X \rightarrow Y$ be a projective morphism from a nonsingular variety $X$ to a variety $Y$, $D = \sum D_i$ a simple normal crossings divisor on $X$, $A$ an $f$-ample ${\mathbb Q}$-divisor, and
$F: = \lceil A\rceil - A$.  Then we have:
\begin{equation*}
\mathrm{R}^if_*(\Omega_X^j(\log(D))(F + A - G)) = \mathrm{R}^if_*(\Omega_X^j(\log(D))(\lceil A\rceil - G)) = 0
\text{ for } i + j > \dim X.
\end{equation*}
\end{corollary}

\subsection{Failure of a stronger version}

Consider the statement of the Kawamata-Viehweg Vanishing Theorem:
$$\rH^i(X,{\mathcal O}_X(K_X + B + F + A)) = \rH^i(X,{\mathcal O}_X(K_X + B + \lceil A\rceil) = 0 \text{ for }i > 0.$$
In this case, we observe that the only conditions on $B$ and $F$ are:
\begin{enumerate}
\item[(i)] $\text{Supp}(B \cup F)$ is a simple normal crossings divisor, and
\item[(ii)] $B$ and $F$ share no common components.
\end{enumerate}

\begin{remark} 
Suppose $B$ and $F$ had a common component. If this common component is locally defined by $\{x = 0\}$, then $K_X + B + F$ has a local generator of the form $\dfrac{dx}{x^{1 + \delta}} \wedge \cdots$ with $\delta > 0$.  However, this violates the standard philosophy that, in an appropriate logarithmic formulation, one should have no worse than simple poles (i.e., $dx/x^1 = d(\log x)$).
\end{remark}

Note that setting $j = \dim X$ and $B:= D -G$, Theorem \ref{mainvan} implies the Kawamata-Viehweg Vanishing Theorem recalled above. In this case, condition (ii) above follows from the condition $F \leq G \leq D$.\\

On the other hand, still staying in line with the above philosophy, one could imagine the following stronger formulation of Theorem \ref{mainvan}. Let $X$ be as before, $D = \sum_i D_i$ a simple normal crossings divisor on $X$, $A$ an ample $\bbQ$-divisor, and $F:= \lceil A\rceil - A$ such that $\text{Supp}(D \cup F)$ is also a simple normal crossings divisor.  Let $G$ be an integral divisor such that \text{$D \cap F \leq G \leq D$}, where $D \cap F := \sum_{D_i \subset \text{Supp}(F)}D_i$. Then one is led to consider the following stronger vanishing statement where we do not require $F$ to be contained in $G$ or $D$:
\begin{equation*}
\rH^i(X,\Omega_X^j(\log(D))(F + A - G)) = \rH^i(X,\Omega_X^j(\log(D))(\lceil A\rceil - G) = 0
\text{ for } i + j > \dim X.
\end{equation*}

Note that, if $j = \dim X$, then this stronger formulation is actually equivalent to Theorem \ref{mainvan}. Moreover, in this case, they are also both equivalent to the Kawamata-Viehweg Vanishing Theorem.\\

On the other hand, if $j < \dim X$, then the stronger formulation above differs from Theorem \ref{mainvan}.
In fact, if $D=0$, then the stronger formulation would imply that
$$\rH^i(X,\Omega_X^j(\lceil A\rceil)) = 0 \text{ for }i + j > \dim X.$$
In view of the following statement of the Kawamata-Viehweg Vanishing (without any integral part $B$ of the boundary divisor)
$$\rH^i(X,\Omega_X^{\dim X}(\lceil A\rceil)) = \rH^i(X,{\mathcal O}_X(K_X + \lceil A\rceil)) = 0 \text{ for }i > 0,$$
this could be interpreted as a Kawamata-Viehweg type formulation of the Akizuki-Nakano Vanishing Theorem.
However, this naive formulation, as well as the afore-mentioned stronger formulation, fails to hold!\\

In fact, one can also consider the following relative version of the stronger formulation with $D = 0$:
$$\rR^if_*\Omega_X^j(\lceil A\rceil) = 0 \text{ for }i + j > \dim X,$$
where $f:X \rightarrow Y$ is a projective morphism, $A$ is an $f$-ample ${\mathbb Q}$-divisor, and where $F = \lceil A\rceil - A$ is a simple normal crossings divisor on $X$.
However, the following example demonstrates that this statement fails.\\

\begin{example}\label{cexamp1}
Let $Y$ be a non-singular 3-fold, $f:X \rightarrow Y$ be the blow up of a point $P \in Y$, and $E := f^{-1}(P)$ be the exceptional divisor. Then $A = - \epsilon E$ is an $f$-ample ${\mathbb Q}$-divisor for some sufficiently small and positive rational number $0 < \epsilon << 1$. According to the stronger formulation, we should have
$$\rR^2f_*\Omega_X^2(\lceil A\rceil) = \rR^2f_*\Omega_X^2 = 0.$$
On the other hand, we have an exact sequence of coherent ${\mathcal O}_X$-modules
$$0 \longrightarrow {\mathcal K}\ \longrightarrow \Omega_X^2 \overset{\phi}\longrightarrow \Omega_E^2 \rightarrow 0,$$
where $\phi$ is the restriction map and ${\mathcal K}$ is the kernel of the map $\phi$.  The associated long exact sequence gives
$$\rR^2f_*\Omega_X^2 \longrightarrow \rR^2f_*\Omega_E^2 \cong \rH^2({\mathbb P}^2,\Omega_{{\mathbb P}^2}^2) \longrightarrow \rR^3f_*{\mathcal K} = 0.$$
Here the last term vanishes as the fibers of $f$ have dimension at most $2$.
Since by the Serre duality 
$$\rH^2({\mathbb P}^2,\Omega_{{\mathbb P}^2}^2)  \cong \rH^0({\mathbb P}^2, {\mathcal O}_{{\mathbb P}^2}) \cong {\mathbb C} \neq 0,$$
we conclude that
$$\rR^2f_*\Omega_X^2 \neq 0.$$
\end{example}
\subsection{Replacing the condition of $A$ being ample with nef and big}

The statement of the Kodaira Vanishing holds even if we replace an ample divisor $A$ with a nef and big divisor $L$:
$$\rH^i(X,{\mathcal O}_X(K_X + L)) = 0 \text{ for }i > 0,$$
where $X$ is a nonsingular projective variety and $L$ is an (integral) nef and big divisor on $X$.  \\

The proof of this statement via the Kawamata-Viehweg Vanishing for a klt pair $(X,\Delta)$ (``klt'' is short for ``Kawamata log terminal'' singularities) goes as follows.  Since $L$ is big, by the so-called Kodaira Lemma, we can write $L$ as a ${\mathbb Q}$-divisor
$$L = M + H,$$
where $M$ is an effective divisor and $H$ is an ample divisor.  For $n \in {\mathbb N}$, we have another equation of ${\mathbb Q}$-divisors:
$$L = \dfrac{1}{n}\{L + (n - 1)L\} = \dfrac{1}{n}\{M + H + (n - 1)L\} = \dfrac{1}{n}M + \dfrac{1}{n}\{H + (n - 1)L\}.$$
Here $A := \dfrac{1}{n}\{H + (n - 1)L\}$ is an ample ${\mathbb Q}$-divisor, and \text{the pair $(X, \Delta = \dfrac{1}{n}M)$} is klt for $n$ sufficiently large.  As an application of the Kawamata-Viehweg Vanishing Theorem to the klt pair $(X,\Delta)$ we obtain:
$$\rH^i(X,{\mathcal O}_X(K_X + L)) = \rH^i(X,{\mathcal O}_X(K_X + \Delta + A)) = 0 \text{ for }i > 0.$$
Note that in the original setting with the SNC divisor $F = \lceil A\rceil - A$, the klt pair we consider is $(X,\Delta = F)$, and that we obtain
$$\begin{array}{rcl}
\rH^i(X,{\mathcal O}_X(K_X + \lceil A\rceil)) &=& \rH^i(X,{\mathcal O}_X(K_X + F + A)) \\
&=& \rH^i(X,{\mathcal O}_X(K_X + \Delta + A)) = 0 \text{ for }i > 0.\\
\end{array}$$\\

However, it is well-known that the Akizuki-Nakano Vanishing fails if we replace, in its formulation, an ample divisor $A$ with a nef and big divisor $L$ (cf. 4.3.4 \cite{Lazarsfeld}, see also examples \ref{cexamp2.3.1}, \ref{cexamp2.3.2}, \ref{cexamp2.3.3} below.). In particular, there is an example where we have
$$\rH^i(X,\Omega_X^j(L)) \neq 0 \text{ and }i + j > \dim X,$$
where $X$ is a nonsingular projective variety over ${\mathbb C}$ and $L$ is an integral nef and big divisor on $X$.  One might consider this to be a ``pathology'' if one expects that the Akizuki-Nakano Vanishing for a klt pair $(X,\Delta)$ should hold, and hence that one should have for a nef and big divisor $L = \Delta + A$ as above
$$\begin{array}{rcl}
\rH^i(X,\Omega_X^j(L)) &=& \rH^i(X,\Omega_X^j(\Delta + A)) \\
&=& \rH^i(X,\Omega_X^j(\lceil A\rceil)) = 0 \text{ and }i + j > \dim X.\\
\end{array}$$
However, this is exactly the statement of the stronger version of our main theorem discussed above, which we saw fails to hold. Therefore, in the above sense, we may say that the failure of the Akizuki-Nakano Vanishing for a nef and big divisor and the failure of the stronger version of its Kawamata-Viehweg type formulation share the same origin.\\

\begin{example}\label{cexamp2.3.1}
Let $f:X \rightarrow Y$ be the blow up of a point $P \in Y$ on a nonsingular 3-fold $Y$ as in Example \ref{cexamp1}.  Let $L = \pi^*H$ be the pull-back of an ample divisor $H$ on $Y$.  In this case, $L$ is nef and big. Then
$$\rR^2f_*(\Omega_X^2(L)) = \rR^2f_*(\Omega_X^2(\pi^*H)) \cong \rR^2f_*(\Omega_X^2) \otimes H \neq 0,$$
since $\rR^2f_*\Omega_X^2 \neq 0$. In particular, this shows the failure of the (relative) Akizuki-Nakano Vanishing, when we replace an ample divisor $A$ with a nef and big divisor $L$.
\end{example}

\begin{example}\label{cexamp2.3.2}
In the previous example, we can also take $L$ to be the structure sheaf ${\mathcal O}_X$, which is $f$-nef and $f$-big.  Then we have
$$\rR^2f_*(\Omega_X^2(L)) = \rR^2f_*(\Omega_X^2) \neq 0.$$
In particular, this shows the failure of the (relative) Akizuki-Nakano Vanishing, when we replace a relative ample divisor $A$ with a relative nef and big divisor $L$.
\end{example}

\begin{example}\label{cexamp2.3.3} Let $f:X \rightarrow Y$ be as in Example \ref{cexamp2.3.1}.  Consider an ample divisor $H$ on $Y$ and, a sufficiently small and positive rational number $0 < \epsilon << 1$ such that $A = \pi^*H - \epsilon E$ is ample on $X$.  Then looking at the Leray spectral sequence, one immediately sees that
$$\rH^2(X,\Omega_X^2(\lceil A\rceil)) = \rH^2(X,\Omega_X^2(L)) \neq 0.$$
This provides counter-examples to the stronger version of our formulation and the Akizuki-Nakano Vanishing for a nef and big line bundle in the absolute setting.
\end{example}
\subsection{Special cases of the main vanishing result}\
We discuss various special cases of Theorem \ref{mainvan}.

\begin{case} $A$ is integral and $G = 0$.\\
This case yields the Esnault-Viehweg Vanishing Theorem
$$\rH^i(X,\Omega_X^j(\log(D))( A)) = 0 \text{ for }i + j > \dim X.$$
When $j = \dim X$, it yields the logarithmic version of \text{the Kodaira Vanishing Theorem.}
\end{case}

\begin{case} $j = \dim X$\\
By setting $B = D - G$, this case yields the Kawamata-Viehweg Vanishing Theorem:
$$
\rH^i(X,\Omega_X^{\dim X}(D + \lceil A\rceil - G)) = \rH^i(X,{\mathcal O}_X(K_X + B + F + A)) = 0 \text{ for }i > 0.$$
Here $B$ and $F$ share no common components because of the condition $F \leq G \leq D$.
\end{case}

\begin{case} $G = D$.\\
This case yields
$$\begin{array}{rcl}
\rH^i(X,\Omega_X^j(\log(D))(F + A - D)) &=& \rH^i(X,\Omega_X^j(\log(D))(\lceil A\rceil - D) )\\
&=& 0 \hskip.4in \text{ for }i + j > \dim X.\\
\end{array}$$
This is the fractional version of the Steenbrink Vanishing Theorem, which appears in \cite{Arapura2}. We note that when $j = \dim X$ and $A$ is integral, we recover the Kodaira Vanishing Theorem, but not its logarithmic version (unless we use the round up trick \ref{helmke}).
\end{case}
\begin{case} $D = G = E$, where $E$ is the support of a projective birational map $f: X \rightarrow Y$.\\
Consider a projective birational map $f: X \rightarrow Y$ from a nonsingular variety $X$. Then Corollary \ref{relmainvan} implies
$$\rR^if_*\Omega^j_X(\log(E))(\lceil A\rceil - E) = \rR^if_*\Omega^j_X(\log(E))(- E) = 0 \text{ for }i + j > \dim X$$
where $E = \sum E_i$ is the exceptional divisor (which is assumed to be a simple normal crossings divisor), $A = \sum - e_iE_i$ is an $f$-ample divisor with $\lceil A\rceil = 0$.  When $j = \dim X$, the statement becomes
$$R^if_*\omega_X = 0 \text{ for }i > 0,$$
which is nothing but the Grauert-Riemenschneider Vanishing Theorem.
\end{case}

\section{Proof of Theorem \ref{mainvan} by Kawamata Covering Lemma}

In this section, we provide a proof of Theorem \ref{mainvan} using the Kawamata Covering Lemma.

\subsection{The case when $A$ integral}
In this subsection, we prove Theorem \ref{mainvan} in the setting where $A$ is integral. We shall further split this case into subcases.

\begin{subcase} $G = D$.\\
In this case, the statement is nothing but the Steenbrink Vanishing Theorem.  Note that, when $G = D = 0$, we obtain the Akizuki-Nakano Vanishing Theorem.
\end{subcase}

\begin{subcase} $G \leq D' = D - D_1  = D_2 + \cdots + D_l <  D = D_1 + D_2 + \cdots + D_l$.\\
In this case, one proceeds via induction on the number of the components in $D$ and the dimension of $X$.  
Consider the residue sequence
$$
0 \rightarrow \Omega_X^j(\log(D'))(A - G)  
\rightarrow \Omega_X^j(\log(D))(A - G) 
\xrightarrow{\psi} \Omega_{D_1}^{j-1}(\log(D'|_{D_1}))((A - G)|_{D_1}) 
\rightarrow 0,\\
$$
where $\psi$ is the residue map. 
\begin{comment}
Recall, on local generators, $\psi$ is defined as follows:
$$\left\{\begin{array}{l}
\dfrac{dx_1}{x_1} \wedge_{s \in S} \dfrac{dx_s}{x_s} \wedge_{t \in T} dx_t \overset{\psi}\longrightarrow \wedge_{s \in S} \dfrac{dx_s}{x_s} \wedge_{t \in T} dx_t \\
\text{where } \{x_s = 0\} \subset D' \hskip.05in \forall s \in S \\
\text{and where }\{x_t = 0\} \not\subset D \hskip.05in \forall t \in T \text{ with } 1 + \# S + \# T = j, \\
\wedge_{s \in S} \dfrac{dx_s}{x_s} \wedge_{t \in T} dx_t \hskip.3in \overset{\psi}\longrightarrow 0 \\
\text{where } \{x_s = 0\} \subset D' \hskip.05in \forall s \in S \\
\text{and where } \{x_t = 0\} \not\subset D \hskip.05in \forall t \in T \text{ with } \# S + \# T = j. \\
\end{array}\right.$$
\end{comment}
The corresponding long exact sequence in cohomology gives
\begin{align*}
\cdots \rightarrow  \rH^i(X,\Omega_X^j(\log(D'))(A - G)) 
\rightarrow  \rH^i(X, \Omega_X^j(\log(D))(A - G)) \\
\rightarrow  \rH^i(D_1,\Omega_{D_1}^{j-1}(\log(D'|_{D_1}))((A - G)|_{D_1})) \rightarrow \cdots.
\end{align*}

If $i + j > \dim X$, then the first term is $0$ by induction on the number of the components in $D$ (since the number of the components in $D'$ is one less than that of $D$). On the other hand, if $i + j > \dim X$, the last term is also $0$ by induction on the dimension of $X$ (since $\dim D_1 = \dim X - 1$ and since $i + (j-1) = i + j - 1 > \dim X - 1 = \dim D_1$).  Therefore, we conclude that
$$\rH^i(X, \Omega_X^j(\log(D))(A - G)) = 0$$
if $i + j > \dim X$.
\begin{remark} 
Note that using the residue sequence above with $G = 0$, one can derive the Esnault-Viehweg Vanishing from the Akizuki-Nakano Vanishing via induction on the number of the components in $D$ and the dimension of $X$.  But, it seems that the Steenbrink Vanishing cannot be derived from the Akizuki-Nakano Vanishing via a simple inductive argument using the residue sequence above.
\end{remark}
\end{subcase}

\subsection{The case when $A$ fractional}
We reduce the case where $A$ is fractional to the case where $A$ is integral, using the following Kawamata Covering Lemma.
\begin{lemma}[{\bf Kawamata Covering Lemma}(\cite{Kawamata,KMM,Matsuki}]\label{covlem}
There exists a finite morphism $\pi:Y \rightarrow X$ with the extension of the function fields ${\mathbb C}(Y)/{\mathbb C}(X)$ being Galois (and hence $\Gamma := \mathrm{Gal}({\mathbb C}(Y)/{\mathbb C}(X))$ acts on $Y$ over $X$) such that:\begin{enumerate}
\item[(i)] $Y$ is nonsingular projective.
\item[(ii)] $\pi^*A$ is integral.
\item[(iii)] $\pi$ is ramifed only along $D \cup M$, which forms an SNC divisor for some auxiliary divisor $M$, with $D$ and $M$ sharing no common components.
\item[(iv)] There is a sufficiently divisible and large integer $m \in {\mathbb N}$ such that, for any irreducible component $B$ in $D \cup M$, we have 
$$\pi^*B = mB_Y$$ 
where $B_Y = \pi^{-1}(B)_{\text{red}}$ and that, if $B \subset F$, we have
$$(\star)\ a_B + \dfrac{m - 1}{m} \geq \lceil a_B\rceil.$$
Here $a_B$ is the coefficient of $B$ in $A$.
\item[(v)] For any closed point $P \in X$ there exists a regular system of parameters \linebreak $(x_1, \ldots, x_l, x_{l+1}, \ldots, x_n)$ such that 

$\bullet$ $\{\prod_{\alpha = 1}^l x_{\alpha} = 0\} = (D \cup M)_P$, and 

$\bullet$ any closed point $Q \in \pi^{-1}(P)$ has a regular system of paramaters of the form $(y_1 = x_1^{1/m}, \ldots, y_l = x_l^{1/m}, x_{l+1}, \ldots, x_n)$ \text{(for the same integer ``$m$'' mentioned in condition (iv)).}
\end{enumerate}
\end{lemma}

\begin{lemma}\label{Ginv}  With notation as in Lemma \ref{covlem}, we have

$$\left[\pi_*\{\Omega^j(\mathrm{log}(D_Y))(\pi^*A - G_Y)\}\right]^{\Gamma} = \Omega^j(\mathrm{log}(D))(\lceil A\rceil - G).$$

\end{lemma}

\begin{proof}

First note that
$$\begin{array}{rcl}
\Omega_Y^j(\mathrm{log}(D_Y))(\pi^*A - G_Y) &\subset& \Omega_Y^j(\mathrm{log}((D \cup M)_Y)) \otimes_{{\mathcal O}_Y}{\mathbb C}(Y) \\
&=& \pi^*\{\Omega_X^j(\mathrm{log}(D \cup M))\} \otimes_{{\mathcal O}_Y}{\mathbb C}(Y) \\
\end{array}$$
and hence that
$$\pi_*\{\Omega_Y^j(\mathrm{log}(D_Y))(\pi^*A - G_Y)\} \subset \Omega_X^j(\mathrm{log}(D \cup M)) \otimes_{{\mathcal O}_X} {\mathbb C}(Y).$$
The $\Gamma$-action on the left-hand side is induced from the $\Gamma$-action on the right-hand side, where $\Gamma$ acts trivially on the first factor $\Omega_X^j(\mathrm{log}(D \cup M))$ and $\Gamma$ acts on the second factor ${\mathbb C}(Y)$ as the Galois group $\mathrm{Gal}({\mathbb C}(Y)/{\mathbb C}(X))$.  Therefore, we conclude
$$\left[\pi_*\{\Omega_Y^j(\mathrm{log}(D_Y))(\pi^*A - G_Y)\}\right]^{\Gamma} \subset \Omega_X^j(\mathrm{log}(D \cup M)) \otimes_{{\mathcal O}_X} {\mathbb C}(X).$$
Our task is to identify the left-hand side with another subsheaf of the right-hand side
$$\Omega_X^j(\mathrm{log}(D))(\lceil A\rceil - G) \subset \Omega_X^j(\mathrm{log}(D \cup M)) \otimes_{{\mathcal O}_X} {\mathbb C}(X).$$

For a closed point $P \in X$, we choose a regular system of parameters 
$$(\{x_s\}_{s \in S},\{x_t\}_{t \in T}, \{x_v\}_{v \in V}, \{x_w\}_{w \in W}, \{x_z\}_{z \in Z})$$
as in condition (v) of the Kawamata Covering Lemma and an affine open neighborhood $P \in U$ such that
$$\left\{\begin{array}{lcl}
\{x_s = 0\}_{s \in S} &=& (D \setminus G)_P = (D \setminus G)|_U \\
\{x_t = 0\}_{t \in T} &=& (G \setminus F)_P = (G \setminus F)|_U \\
\{x_v = 0\}_{v \in V} &=& F_P = F|_U \\
\{x_w = 0\}_{w \in W} &=& M_P = M|_U \\
\{x_z = 0\}_{z \in Z} &&\text{shares no components with }(D \cup M)_P \text{ or }(D \cup M)|_U,\\
\end{array}\right.$$
and that
$$\left\{\bigwedge_{s \in S_{\alpha} \subset S}\dfrac{dx_s}{x_s} \bigwedge_{t \in T_{\beta} \subset T}\dfrac{dx_t}{x_t} \bigwedge_{v \in V_{\gamma} \subset V}\dfrac{dx_v}{x_v} \bigwedge_{w \in W_{\delta} \subset W}\dfrac{dx_w}{x_w} \bigwedge_{z \in Z_{\epsilon} \subset Z}dx_z\right\},$$
where the collection of the subsets $S_{\alpha} \subset S, T_{\beta} \subset T, V_{\gamma} \subset V, W_{\delta} \subset W, Z_{\epsilon} \subset Z$ is the one of all those with $\# S_{\alpha} + \# T_{\beta} + \# V_{\gamma} + \# W_{\delta} + \# Z_{\epsilon} = j$, forms a basis of $\Omega_X^j(\log(D \cup M))$ as a free ${\mathcal O}_X$-module over $U$, while
$$\left\{\bigwedge_{s \in S_{\alpha} \subset S}\dfrac{dx_s}{x_s} \bigwedge_{t \in T_{\beta} \subset T}\dfrac{dx_t}{x_t} \bigwedge_{v \in V_{\gamma} \subset V}\dfrac{dx_v}{x_v} \bigwedge_{w \in W_{\delta} \subset W}dx_w \bigwedge_{z \in Z_{\epsilon} \subset Z}dx_z\right\}$$
forms a basis of $\Omega_X^j(\log(D))$ as a free ${\mathcal O}_X$-module over $U$.

Since $\pi$ ramifies only over $D \cup M$, we conclude that
$$\begin{array}{l}
\left\{\pi^*\left[\bigwedge_{s \in S_{\alpha} \subset S}\dfrac{dx_s}{x_s} \bigwedge_{t \in T_{\beta} \subset T}\dfrac{dx_t}{x_t} \bigwedge_{v \in V_{\gamma} \subset V}\dfrac{dx_v}{x_v} \bigwedge_{w \in W_{\delta} \subset W}\dfrac{dx_w}{x_w} \bigwedge_{z \in Z_{\epsilon} \subset Z}dx_z\right] \right\} \\
= \left\{\bigwedge_{s \in S_{\alpha} \subset S}\dfrac{dy_s}{y_s} \bigwedge_{t \in T_{\beta} \subset T}\dfrac{dy_t}{y_t} \bigwedge_{v \in V_{\gamma} \subset V}\dfrac{dy_v}{y_v} \bigwedge_{w \in W_{\delta} \subset W}\dfrac{dy_w}{y_w} \bigwedge_{z \in Z_{\epsilon} \subset Z}dx_z \right\} \\
\end{array}$$
forms a basis of $\Omega_Y^j(\log((D\cup M)_Y))$ as a free ${\mathcal O}_Y$-module over $\pi^{-1}(U)$, while 
$$\left\{\bigwedge_{s \in S_{\alpha} \subset S}\dfrac{dy_s}{y_s} \bigwedge_{t \in T_{\beta} \subset T}\dfrac{dy_t}{y_t} \bigwedge_{v \in V_{\gamma} \subset V}\dfrac{dy_v}{y_v} \bigwedge_{w \in W_{\delta} \subset W}dy_w \bigwedge_{z \in Z_{\epsilon} \subset Z}dx_z \right\}$$
forms a basis of $\Omega_Y^j(\log((D)_Y))$ as a free ${\mathcal O}_Y$-module over $\pi^{-1}(U)$.

Take a section
$$\begin{array}{rcl}
f &\in& \Gamma(\pi^{-1}(U), \Omega_Y^j(\mathrm{log}((D \cup M)_Y)) \otimes_{{\mathcal O}_Y}\bbC(Y)) \\
&=& \Gamma(\pi^{-1}(U), \pi^*\{\Omega_X^j(\mathrm{log}(D \cup M))\} \otimes_{{\mathcal O}_Y}\bbC(Y)) \\
\end{array}$$
and write
$$f = \sum_{\alpha, \beta, \gamma, \delta, \epsilon} \left(\pi^*\left[\bigwedge_{s \in S_{\alpha} \subset S}\dfrac{dx_s}{x_s} \bigwedge_{t \in T_{\beta} \subset T}\dfrac{dx_t}{x_t} \bigwedge_{v \in V_{\gamma} \subset V}\dfrac{dx_v}{x_v} \bigwedge_{w \in W_{\delta} \subset W} \dfrac{dx_w}{x_w} \bigwedge_{z \in Z_{\epsilon} \subset Z}dx_z\right] \otimes f_{\alpha,\beta,\gamma,\delta,\epsilon}\right)$$
with $f_{\alpha,\beta,\gamma,\delta,\epsilon} \in {\mathbb C}(Y)$.  

\vskip.03in

Observe
$$\begin{array}{rcl}
f &\in& \Gamma(U, \pi_*\{\Omega_Y^j(\mathrm{log}(D_Y))(\pi^*A - G_Y)\}) \\
&=& \Gamma(\pi^{-1}(U), \Omega_Y^j(\mathrm{log}(D_Y)(\pi^*A - G_Y)) \\
&\Longleftrightarrow& \text{div}\left(f_{\alpha,\beta,\gamma,\delta,\epsilon}/\prod_{w \in W_{\delta} \subset W}y_w\right) + \pi^*A - G_Y|_{\pi^{-1}(U)} \geq 0, \\
&& f_{\alpha,\beta,\gamma,\delta,\epsilon} \in {\mathbb C}(Y), \forall \alpha, \beta, \gamma, \delta, \epsilon \\
&\Longleftrightarrow& \text{div}\left(f_{\alpha,\beta,\gamma,\delta,\epsilon} /\prod_{w \in W_{\delta} \subset W}\pi^*(x_w)^{1/m}\right) + \pi^*A - G_Y|_{\pi^{-1}(U)} \geq 0, \\
&& f_{\alpha,\beta,\gamma,\delta,\epsilon} \in {\mathbb C}(Y), \forall \alpha, \beta, \gamma, \delta, \epsilon.\\
\end{array}$$
Therefore, we conclude
$$\begin{array}{rcl}
f &\in& \Gamma(U, \left[\pi_*\{\Omega_Y^j(\mathrm{log}(D_Y))(\pi^*A - G_Y)\}\right]^{\Gamma}) \\
&\Longleftrightarrow& \text{div}\left(f_{\alpha,\beta,\gamma,\delta,\epsilon}/\prod_{w \in W_{\delta} \subset W}\pi^*(x_w)^{1/m}\right) + \pi^*A - G_Y|_{\pi^{-1}(U)} \geq 0, \\
&& f_{\alpha,\beta,\gamma,\delta,\epsilon} \in {\mathbb C}(Y)^\Gamma = {\mathbb C}(X), \forall \alpha, \beta, \gamma, \delta, \epsilon\\
&\Longleftrightarrow& \text{div}\left(f_{\alpha,\beta,\gamma,\delta,\epsilon}/\prod_{w \in W_{\delta} \subset W}\pi^*(x_w)^{1/m}\right) + \pi^*A - \pi^*(G) + \pi^*\left(\dfrac{m - 1}{m}G\right)|_{\pi^{-1}(U)} \geq 0, \\
&& f_{\alpha,\beta,\gamma,\delta,\epsilon} \in {\mathbb C}(Y)^\Gamma = {\mathbb C}(X), \forall \alpha, \beta, \gamma, \delta, \epsilon\\
&\Longleftrightarrow& \text{div}\left(f_{\alpha,\beta,\gamma,\delta,\epsilon} / \prod_{w \in W_{\delta} \subset W}x_w^{1/m}\right) + A - G + \dfrac{m - 1}{m}G|_U \geq 0, \\
&& f_{\alpha,\beta,\gamma,\delta,\epsilon} \in {\mathbb C}(X), \forall \alpha, \beta, \gamma, \delta, \epsilon\\
&\Longleftrightarrow& \text{div}\left(f_{\alpha,\beta,\gamma,\delta,\epsilon} / \prod_{w \in W_{\delta} \subset W}x_w\right) + \lceil A\rceil - G|_U \geq 0, \\
&& f_{\alpha,\beta,\gamma,\delta,\epsilon} \in {\mathbb C}(X), \forall \alpha, \beta, \gamma, \delta, \epsilon\\
&\Longleftrightarrow& \\
f &\in& \Gamma(U, \Omega^j(\mathrm{log}(D))(\lceil A\rceil - G)). \\
\end{array}$$
We may add the following explanation for the second last equivalence: Let $B$ vary among all the irreducible components of $D \cup M|_U$.  Then the  3rd last condition
$$\text{div}\left(f_{\alpha,\beta,\gamma,\delta,\epsilon} / \prod_{w \in W_{\delta} \subset W}x_w^{1/m}\right) + A - G + \dfrac{m - 1}{m}G|_U \geq 0$$
reads for the component $B$:
$$\left\{\begin{array}{cll}
\bullet &v_B(f_{\alpha,\beta,\gamma,\delta,\epsilon}) + a_B - 0 + \dfrac{m-1}{m} \cdot 0 \geq 0 & \\
\Longleftrightarrow & v_B(f_{\alpha,\beta,\gamma,\delta,\epsilon}) + a_B - 0 = v_B(f_{\alpha,\beta,\gamma,\delta,\epsilon}) + \lceil a_B\rceil - 0 \geq 0 &\text{ if } B \subset D \setminus G \\
a_B \in {\mathbb Z}&& \\
\bullet &v_B(f_{\alpha,\beta,\gamma,\delta,\epsilon}) + a_B - 1 + \dfrac{m-1}{m} \cdot 1 \geq 0 & \\
\Longleftrightarrow & v_B(f_{\alpha,\beta,\gamma,\delta,\epsilon}) + a_B - 1 = v_B(f_{\alpha,\beta,\gamma,\delta,\epsilon}) + \lceil a_B\rceil - 1 \geq 0 &\text{ if } B \subset G \setminus F \\
a_B \in {\mathbb Z}&& \\
\bullet &v_B(f_{\alpha,\beta,\gamma,\delta,\epsilon}) + a_B - 1 + \dfrac{m-1}{m} \cdot 1 & \\
=&  v_B(f_{\alpha,\beta,\gamma,\delta,\epsilon}) + a_B + \dfrac{m-1}{m} \cdot 1 - 1 \geq 0 &  \\
\Longleftrightarrow & v_B(f_{\alpha,\beta,\gamma,\delta,\epsilon}) + \lceil a_B\rceil - 1 \geq 0 & \text{ if } B \subset F \\
\text{by condition }(\star)&& \\
\bullet &v_B(f_{\alpha,\beta,\gamma,\delta,\epsilon}) - \dfrac{1}{m} + a_B - 0 + \dfrac{m-1}{m} \cdot 0 \geq 0 & \\
\Longleftrightarrow & v_B(f_{\alpha,\beta,\gamma,\delta,\epsilon}) - 1 +  a_B - 0 = v_B(f_{\alpha,\beta,\gamma,\delta,\epsilon}) - 1 +  \lceil a_B\rceil - 0 \geq 0 & \text{ if } B \subset M \\
a_B \in {\mathbb Z}&&   \\
\end{array}\right.$$
where $A = \sum a_BB$.

\end{proof}

\vskip.1in

Now Theorem \ref{mainvan} in the fractional case is an immediate consequence of \ref{Ginv} as follows:
$$\begin{array}{rcl}
\rH^i(X,\Omega_X^j(\mathrm{log}(D))(\lceil A\rceil - G)) &=& \rH^i\left(X,\left[\pi_*\{\Omega_Y^j(\mathrm{log}(D_Y))(\pi^*A - G_Y)\}\right]^\Gamma\right) \\
&=& \rH^i(X,\pi_*\{\Omega_Y^j(\mathrm{log}(D_Y))(\pi^*A - G_Y)\})^\Gamma \\
&=& \rH^i(Y,\Omega_Y^j(\mathrm{log}(D_Y))(\pi^*A - G_Y))^\Gamma \\
&=& 0, \\
\end{array}$$
since we have
$$ \rH^i(Y,\Omega^j(\mathrm{log}(D_Y))(\pi^*A - G_Y)) = 0 \text{ for }i + j > \dim X = \dim Y,$$
using the vanishing statement for the case where $\pi^*A$ is integral.
This completes the proof of the main theorem in the case where $A$ is fractional.

\subsection{Some remarks on the first proof}

\subsubsection{Basic Idea of the proof}
If we pretend that $\pi$ is ramified only over $D$, then the idea of the proof for \ref{Ginv} is more transparent. Under the pretension, since the logarithmic differential forms do not ramify and $G_Y = \dfrac{1}{m}\pi^*G = \dfrac{m-1}{m}\pi^*G - \pi^*G$, we have
$$\Omega_Y^j(\mathrm{log}(D_Y))(\pi^*A - G_Y) = \pi^*\{\Omega_X^j(\mathrm{log}(D))\}\left(\pi^*A + \dfrac{m-1}{m}\pi^*G - \pi^*G \right).$$
By taking $\pi_*$ and the $\Gamma$-invariant part, we conclude
$$\begin{array}{rcl}
\pi_*\{\Omega_Y^j(\mathrm{log}(D_Y))(\pi^*A - G_Y)\}^\Gamma &=& \pi_*\left[\pi^*\left\{\Omega_X^j(\mathrm{log}(D))\right\}\left(\pi^*A + \dfrac{m-1}{m}\pi^*G - \pi^*G\right) \right]^\Gamma \\
&=& \Omega_X^j(\mathrm{log}(D))\left(A + \dfrac{m-1}{m}G - G\right) \\
&=& \Omega_X^j(\mathrm{log}(D))\left(\lceil A\rceil - G\right).\\
\end{array}$$
Here the last equality, replacing $A + \dfrac{m-1}{m}G$ with $\lceil A\rceil$, results from the fact that only the fractional part of $A$ is affected and, hence that the coefficients of the components exceed their round ups when we add $\dfrac{m-1}{m}G$.\\

In the actual proof without the pretension, we have to analyze in more detail how a basis of the free ${\mathcal O}_X$-module $\Omega_X^j(\mathrm{log}(D))$ ramifies over $M$, when pulled back by $\pi$, compared to a basis of the free ${\mathcal O}_Y$-module $\Omega_Y^j(\mathrm{log}(D_Y))$ (and conclude that the ramification does not affect the conclusion at all).  The basic idea, however, is the same.

\subsubsection{Use of the logarithmic forms and subtraction of the divisor $G$}\label{uselog}

In contrast to the logarithmic differential forms, if we use the usual differential forms, the basis of the free ${\mathcal O}_X$-module $\Omega_X^j$
$$\left\{\bigwedge_{\{x_{\alpha} = 0\} \subset (D \cup M)_P}dx_{\alpha} \bigwedge_{\{x_{\beta} = 0\} \not\subset (D \cup M)_P}dx_{\beta}\right\}$$
has varying ramification factors, and gives rise to the following corresponding basis of the free ${\mathcal O}_Y$-module $\Omega_Y^j$
$$\left\{\prod_{\{x_{\alpha} = 0\} \subset (D \cup M)_P}\dfrac{1}{\pi^*x_{\alpha}^{(m-1)/m}} \cdot \pi^*\left[\bigwedge_{\{x_{\alpha} = 0\} \subset (D \cup M)_P}dx_{\alpha} \bigwedge_{\{x_{\beta} = 0\} \not\subset (D \cup M)_P}dx_{\beta}\right]\right\}.$$
The varying ramifications cannot be expressd by the twist of a single (${\mathbb Q}\text{-}$) divisor.  This is why one is led to the use of logarithmic differential forms.\\

On the other hand, if we use the logarithmic differential forms, since there is no ramification, there is no ``push'' from the ramification to raise $A$ to $\lceil A\rceil$.  This is where the subtraction of the divisor $G$ comes in.  The difference between $- G_Y = - \pi^*G + \dfrac{m-1}{m}\pi^*G$ and $- \pi^*G$, which is $\dfrac{m-1}{m}\pi^*G$, gives the push \text{to raise $A$ to $\lceil A\rceil$.}

\subsubsection{Comparison with the classical argument}

In the classical argument for the proof of the Kawamata-Viehweg Vanishing, where we only have to deal with the top form, the free ${\mathcal O}_X$-module $\Omega_X^{n = \dim X}$ is of rank one, having one generator
$$\bigwedge_{\alpha = 1}^l dx_{\alpha} \bigwedge_{\beta = l + 1}^{n = \dim X}dx_{\beta}.$$
Therefore, it has a unique ramification factor giving rise to the following unique basis of the free ${\mathcal O}_Y$-module $\Omega_Y^n$:
$$\prod_{\alpha = 1}^l\dfrac{1}{\pi^*x_{\alpha}^{(m-1)/m}} \cdot \pi^*\left[\bigwedge_{\alpha = 1}^l dx_{\alpha} \bigwedge_{\beta = l + 1}^{n = \dim X}dx_{\beta}\right].$$
Moreover, the reciprocal $\prod_{\alpha = 1}^lx_{\alpha}^{(m-1)/m}$ of the ramification factor gives the ``push'' to raise $A$ to $\lceil A\rceil$.  However, the classical argument to look at the usual differential forms would face trouble in the case of lower degree forms as the basis has varying ramification factors (as discussed in \ref{uselog}).\\

Our new argument using the logarithmic forms and subtraction of the divisor $G$ applies to the lower differential forms in the setting dealing with the Kawamata-Viehweg type formulation of the (log) Akizuki-Nakano Vanishing as well as to the top differential form in the setting dealing with the classical Kawamata-Viehweg Vanishing. This also gives a slightly different view point towards the classical argument for the Kawamata-Viehweg Vanishing Theorem.

\subsubsection{An Alternative line of argument}\label{helmke} As suggested by Prof. Helmke, one could follow the following line of argument to prove our main vanishing result:

\begin{enumerate}
\item[(1)] Prove the case with $G = D$ of our formulation, i.e., $$\rH^i(X,\Omega_X^j(\log(D))(\lceil A\rceil - D)),$$ reducing its verification to the Steenbrink Vanishing Theorem via the Kawamata Covering Lemma as in our argument above.
\item[(2)]In order to prove the general case $F \leq G \leq D$, we set $A' = A + \epsilon(D - G)$ for a sufficiently small positive number $0 < \epsilon << 1$ so that $A'$ is again ample with $F' = \lceil A'\rceil - A' \leq D$.  Now, using (1), we conclude
$$0 = \rH^i(X,\Omega_X^j(\mathrm{log}(D)(\lceil A'\rceil - D)) = \rH^i(X,\Omega_X^j(\mathrm{log}(D))(\lceil A\rceil - G))$$
as required.
\end{enumerate}
This line of argument avoids the use of the residue sequence.  It also makes it clearer that what is essential is

$\bullet$ the Steenbrink Vanishing, and

$\bullet$ its fractional version as in \cite{Arapura2}.

\section{Proof of Theorem \ref{mainvan} by reduction mod $p$ via the result of Deligne-Illusie \cite{DI}}\label{section:pvanishing}

Here we explain how to obtain the fractional version of Steenbrink Vanishing, mentioned above, by reduction mod $p$.  We prove the following theorem, via the results of Deligne-Illusie and Raynaud \cite{DI} and a lemma by Hara \cite{Hara}, in characteristic $p$.  The following is special case of  \cite[theorem 8.2]{Arapura2}. \\

\begin{theorem}\label{pvanishing} Let $X$ be a nonsingular projective variety  over an algebraically closed field $k$ of characteristic $p>\dim X$. Let
 $D = \sum D_i$ be a simple normal crossings divisor such that the pair $(X,D)$ is liftable modulo $p^2$. If $L$ is a line bundle such that  
 $L(-\Delta)$ is ample
for some  $\bbQ$-divisor  $\Delta$ supported on $D$ with coefficients in $[0,1)$,
then 
\begin{equation*}
\rH^i(X,\Omega_X^j(\log D)(-D) \otimes L)=0  
\end{equation*}
for $i+j>\dim X$.
\end{theorem}

Now by a standard ``spreading out" argument, we obtain the following result in characteristic $0$, which implies the main theorem \ref{mainvan} (as explained in \ref{helmke}).  We note that the line bundle $L$ and the $\mathbb{Q}$-divisor $\Delta$ correspond to $\lceil A\rceil$ and $F$ in the notation of \S 2.

\begin{corollary}
 Let $X,D$, and $L$ be as above, but defined over an algebraically closed field of characteristic $0$. Then
 \begin{equation*}
\rH^i(X,\Omega_X^j(\log D)(-D) \otimes L)=0  
\end{equation*}
for $i+j>\dim X$.
\end{corollary}

We give a short self contained proof (via the results and lemma mentioned above) of \ref{pvanishing}, extracting the ideas from the proof of  \cite[theorem 8.2]{Arapura2}.  First let us quote the following lemma by Hara \cite{Hara}.

\begin{lemma}[Hara {\cite[3.3]{Hara}}]
 If  $D'$ is an integral divisor satisfying $0\leq D'\leq (p-1)D$, then there is a quasi-isomorphism
 $$\Omega_X^{\bullet}(\log D)\cong \Omega_X^\bullet(\log D)(D')$$

\end{lemma}

Now using the results of \cite{DI} and the above lemma, we obtain the following.

\begin{lemma}\label{lemma:boot1}
Let $M$ be a line bundle on $X$. Suppose that  $D'$ is an integral divisor and that $0\le D'\le (p-1)D$.
Then the following inequalities hold.
\begin{enumerate}
\item[(a)] For all $r$,
$$\sum_{i+j=r}h^i(X, \Omega_X^j(\log D) \otimes M)\le \sum_{i+j=r}h^i(X,
\Omega_X^j(\log D)(D')\otimes M^p)$$

\item[(b)] For all $r$,
$$\sum_{i+j=r}h^i(X, \Omega_X^j(\log D)(-D) \otimes M)\le \sum_{i+j=r}h^i(X,
\Omega_X^j(\log D)(-D-D')\otimes M^p)$$

\end{enumerate}

\end{lemma}

\begin{proof}
Let $F:X\to X$ denote the absolute Frobenius map.
  By \cite[\S 4.2]{DI}, the projection formula, and the previous lemma, we have
  \begin{equation*}
    \begin{split}
 \rH^i(X, \bigoplus_j \Omega_X^j(\log D)[-j] \otimes M) &\cong \bbH^i(X,
 (F_*\Omega_X^\bullet(\log D) )\otimes M)     \\
&\cong \bbH^i(X,
 F_*(\Omega_X^\bullet(\log D) \otimes M^p))  \\
&\cong \bbH^i(X,\Omega_X^\bullet(\log D) \otimes M^p)\\
&\cong \bbH^i(X,\Omega_X^\bullet(\log D)(D') \otimes M^p)
  \end{split}
  \end{equation*}
  These isomorphisms  together with  the spectral sequence
  $$E_1^{ab}= \rH^b(X,\Omega_X^a(\log D)(D')\otimes M^p)\Rightarrow
\bbH^{a+b}(X,\Omega_X^\bullet(\log D)(D') \otimes M^p)$$
prove the first inequality.  We obtain the second inequality from the first using the Serre duality.

\end{proof}

\begin{lemma}\label{lemma:boot2} 
 With the same notation as in  the previous lemma for $X, D$ and $M$, suppose this time that  $D'$ is an integral divisor and  that $0\le D'\le (p^n-1)D$.
 Then 
  $$\sum_{i+j=r}h^i(X, \Omega_X^j(\log D)(-D) \otimes M)\le \sum_{i+j=r}h^i(X,
\Omega_X^j(\log D)(-D-D')\otimes M^{p^n})$$
\end{lemma}

\begin{proof}
  We may write $D' = p^{n-1} D_1' + p^{n-2} D_2'+\ldots$, where $0\le D_i'\le (p-1)D$. Then repeatedly 
applying lemma \ref{lemma:boot1} gives
\begin{equation*}
    \begin{split}
    \sum_{i+j=r}h^i(X, \Omega_X^j(\log D)(-D) \otimes M) &\le
    \sum_{i+j=r}h^i(X, \Omega_X^j(\log D)(-D) \otimes M^p(-D_1'))\\
     &\le
    \sum_{i+j=r}h^i(X, \Omega_X^j(\log D)(-D) \otimes M^{p^2}(-pD_1'-D_2'))\\
&\ldots
 \end{split}
  \end{equation*}
 \end{proof}

\begin{proof}[Proof of \ref{pvanishing}]
 By assumption,  $L(-\Delta)$ is ample for some $\Delta=\sum r_iD_i$ with $r_i\in [0,1)\cap \bbQ$.
 Using Kleiman's ampleness criterion (cf \cite{Lazarsfeld}),  we can see that $L(-\sum r_i'D_i)$  remains ample,  whenever
 $r'_i$ is sufficiently close to $r_i$. Therefore, we can assume that the coefficients $r_i$ lie in $[0,1)\cap \bbZ[\frac{1}{p^l}]$ for some sufficiently large integer $l$.
Thus,
 $L^{p^n}(-D')$ is ample for some integer $n>0$ and some integral divisor $0 \le D' = p^n(\sum r_i'D_i) \le (p^n-1)D$.  We may also assume, taking $n$ sufficiently large, that
 $$\rH^i(X, \Omega_X^j(\log D)(-D)\otimes L^{p^n}(-D'))=0$$
 for  all $i>0$ by the Serre Vanishing.  Now \ref{pvanishing} is a consequence of lemma \ref{lemma:boot2}, noting that, if $i + j = r > \dim X$, then either $i > 0$ or $i = 0$ with $j > \dim X$ and hence $\rH^0(X, \Omega_X^j(\log D)(-D)\otimes L) = \rH^0(X, \Omega_X^j(\log D)(-D)\otimes L^{p^n}(-D'))=0$.
\end{proof}

\section{Stacky version}

\cite{MO} gives an interpretation of the Kawamata-Viehweg Vanishing as (an application of) the Kodaira Vanishing for a certain Deligne-Mumford stack.  In the same spirit, our main vanishing result can be interpreted as (an application of) the Steenbrink Vanishing for a certain Deligne-Mumford stack.

\section{Applications/Future directions}

The application of the Kawamata-Viehweg Vanishing Theorem in the Minimal Model Program is one of the most remarkable stories in the modern development of the subject of Algebraic Geometry.  Here we list some of the well-known applications of the Akizuki-Nakano Vanishing and the Steenbrink Vanishing in the hope that our KV-type formulation of the (log) Akizuki-Nakano Vanishing will find some interesting applications in the future.  

\subsection{Unobstructedness of the deformation of Fano manifolds}\ 

\begin{case} (Classical unobstructedness of the deformation of Fano manifolds)\\
Let $X$ be a Fano manifold, i.e., a nonsingular projective variety with $- K_X$ being ample.  Then we have
$$\rH^2(X, T_X) \cong \rH^2(X,\Omega^{\dim X - 1}(- K_X)) = 0,$$
and hence the deformation of the Fano manifold has no obstruction \cite{Mori-Mukai} by the Akizuki-Nakano Vanishing.
\end{case}
\begin{case}(Unobstructedness of the deformation of log ${\mathbb Q}$-Fano manifolds)\\ Let $(X,B+F)$ be a pair consisting of a nonsingular projective variety and an effective ${\mathbb Q}$-divisor $B + F = \sum B_k + \sum f_jF_j \hskip.1in (0 < f_j < 1)$ with the support $D = \sum B_k + \sum F_j$ being a simple normal crossings divisor on $X$.  Assume $(X,B+F)$ is a log ${\mathbb Q}$-Fano manifold, i.e., $- (K_X + B + F)$ is ample.  Then the deformation of the pair $(X,B+F)$ is unobstructed, since
$$\begin{array}{rcl}
\rH^2(X,T_X(- \log(D)) &\cong& \rH^2(X,\Omega^{\dim X - 1}(\log(D))(- (K_X + D)) \\
&=& \rH^2(X, \Omega^{\dim X - 1}(\log(D))(\lceil - (K_X + B + F)\rceil - G)) = 0,\\
\end{array}$$
where $G = \sum F_j$ by our Theorem 2.1.
\end{case}
\subsection{Extension of the Akizuki-Nakano Vanishing to singular varieties, and a theorem by Flenner}

Steenbrink's motivation to prove his vanishing theorem was to give a simple proof of the vanishing theorem of Guillen, Navarro, Pascual and  Puerta \cite{Nav}, which can be considered a natural extension (from a certain point of view) of the Kodaira-Akizuki-Nakano Vanishing to singular varieties involving the du Bois complex.  A very nice application of their vanshing theorem is due to Flenner \cite{Flen}, who proves that the regular $l$-forms on the smooth locus of a singular variety extend to the regular forms on any resolution of singularities for 
$l$ less than the codimension of the singular set minus $1$.  Flenner uses the Steenbrink Vanishing only indirectly; a different argument, where the vanishing is used more explicitly, can be found in \cite{Arapura1}.

\subsection{Kovacs' singular version of the Esnault-Viehweg Vanishing and its application to the study of the family of canonically polarized varieties} Kovacs \cite{Kov2} proves a singular version of the Esnault-Viehweg Vanishing (and others).  As an application of these vanishing theorems, he proves an Arakelov-Parshin type boundedness result for the families of canonically polarized varieties with rational Gorenstein singularities (cf.\cite{Kov1, Viehweg-Zuo}).  See also Kovacs' extension  \cite{Kov3} of the Steenbrink Vanishing Theorem.

\subsection{Analysis of the zeo locus of the (log) 1-forms by Wei, extending the previous results of Hacon-Kovacs \cite{Hacon-Kovacs} and Popa-Schnell \cite{Popa-Schnell}} Wei \cite{Wei1, Wei2} proves that the zero-locus of any global holomorphic log-one-form on a projective log-smooth pair (X, D) of log-general type must be non-empty, using some generalizations of the Kodaira-Saito Vanishing theorem \cite{Saito}.  This is a generalization of the results proved by \cite{Hacon-Kovacs} and \cite{Popa-Schnell}.

\subsection{Generalization in terms of the ``multiplier ideal sheaf''} The Kawamata-Viehweg Vanishing Theorem has a generalization in terms of the ``multiplier ideal sheaf''.  It is an interesting question what the proper definition of a ``multiplier ideal sheaf'' and a generalization would be in the context of our vanishing result.  We will discuss the answer to this question in the subsequent papers.

\vskip.1in

\paragraph{\bf Acknowledgements.} We thank Professors O. Fujino, C. Hacon, S. Helmke, Y. Kawamata, Y. Namikawa, A. Moriwaki, S. Mori, S. Mukai, and M. Nori for invaluable comments, suggestions and support.  We also thank the members of our seminar, P. Coupek, H. Li, and H. Wang for listening to the talks about the subject.

\end{document}